\documentclass[11pt,a4paper,reqno]{amsart}
\usepackage{latexsym}
\usepackage{bbm,amssymb,amsthm,amsmath}
\usepackage{graphicx}
\usepackage{multirow}
\usepackage{caption}
\usepackage{bm}
\usepackage{multirow}
\usepackage{mathtools}  
\usepackage{float}  

\usepackage{xcolor}
\usepackage[backref=page]{hyperref}
\definecolor{darkgreen}{rgb}{0.5,0.25,0}
\definecolor{darkblue}{rgb}{0,0,1}
\definecolor{answerblue}{rgb}{0,0,0.75}
\hypersetup{colorlinks,breaklinks,
linkcolor=darkblue,urlcolor=darkblue,
anchorcolor=darkblue,citecolor=darkblue}

\usepackage{enumitem}

\calclayout

\theoremstyle{plain}
\newtheorem{theorem}{Theorem}[section]	
\newtheorem{lemma}[theorem]{Lemma}
\newtheorem{corollary}[theorem]{Corollary}

\theoremstyle{definition}
\newtheorem{definition}[theorem]{Definition}
\newtheorem{example}[theorem]{Example}

\theoremstyle{remark}
\newtheorem{remark}[theorem]{Remark}

\numberwithin{equation}{section}	

\def\C{\mathbb{C}}
\def\R{\mathbb{R}}    
  
\def\N{\mathbb{N}}

\def\Cr{\mathbb{C}^{r}}

\let\mib=\boldsymbol

\def\mnu{{\mib \nu}}

\def\mp{\mathbf{p}}
\def\mq{\mathbf{q}}

\def\vf{\mathsf{f}}
\def\vu{\mathsf{u}}
\def\vv{\mathsf{v}}

\def\ve{\mathsf{e}}

\def\mA{\mathbf{A}}

\def\mB{\mathbf{B}}

\def\mR{\mathbf{R}}
\def\mT{\mathbf{T}}

\def\lD{\mathcal{D}}
\def\lH{\mathcal{H}}
\def\lV{\mathcal{V}}
\def\lW{\mathcal{W}}

\def\lL{\mathcal{L}}

\def\phi{\mathsf{\varphi}}

\def\dom{\operatorname{dom}}
\def\dim{\operatorname{dim}}
\def\ker{\operatorname{ker}}
\def\ran{\operatorname{ran}}

\def\dup#1#2#3#4{{}_{#1\!}\langle\, #2 , #3 \,\rangle_{#4}} 
\def\scp#1#2{\langle\, #1 \mid #2 \,\rangle}  
\def\iscp#1#2{[\, #1 \mid #2 \,]}  

\begin{document}

\title[$m-$accretive extensions of Friedrichs operators]{$m-$accretive extensions of Friedrichs operators}

\author[Burazin]{K.~Burazin}\address{Kre\v simir Burazin, 
    School of Applied Mathematics and Computer Science, 
    J.~J.~Strossmayer University of Osijek, 
    Trg Ljudevita Gaja 6, 31000 Osijek, Croatia}\email{kburazin@mathos.hr}

\author[Erceg]{M.~Erceg}\address{Marko Erceg,
	Department of Mathematics, Faculty of Science, University of Zagreb, Bijeni\v{c}ka cesta 30,
	10000 Zagreb, Croatia}\email{maerceg@math.hr}

\author[Soni]{S.~K.~Soni}\address{Sandeep Kumar Soni,
	Department of Mathematics, Faculty of Science, University of Zagreb, Bijeni\v{c}ka cesta 30,
	10000 Zagreb, Croatia}\email{sandeep@math.hr}
\curraddr{Institut f\" ur Angewandte Mathematik, TU Graz, Steyrergasse 30/III, 8010 Graz, Austria}
\email{soni@tugraz.at}




\begin{abstract}
The introduction of abstract Friedrichs operators in 2007—an operator-theoretic framework for studying classical Friedrichs operators—has led to significant developments in the field, including results on well-posedness, multiplicity, and classification. More recently, the von Neumann extension theory has been explored in this context, along with connections between abstract Friedrichs operators and skew-symmetric operators.

In this work, we show that all $m-$accretive extensions of abstract Friedrichs operators correspond precisely to those satisfying (V)-boundary conditions. We also establish a connection between the $m-$accretive extensions of abstract Friedrichs operators and their skew-symmetric components. Additionally, three equivalent formulations of boundary conditions are unified within a single interpretive framework. To conclude, we discuss a constructive relation between (V)- and (M)-boundary conditions and examine the multiplicity of the associated $M$-operators. We demonstrate our results by two examples, namely, the first order ordinary differential equation on an interval, with various boundary conditions, and the second-order elliptic partial differential equation with Dirichlet boundary conditions.

\end{abstract}

\maketitle


\section{Introduction}\label{intro}

Ern, Guermond and Caplain \cite{EGC} introduced the concept of \emph{Abstract Friedrichs operators} in 2007, in order to provide a Hilbert space operator-theoretic approach of studying the \emph{classical Friedrichs operators} which were introduced by Friedrichs \cite{KOF} in 1958. Friedrichs' primary motivation was to treat the equations that change their type, like the Tricomi equation, which appear in the transonic fluid flow. The class of classical Friedrichs operators encompasses a wide range of (semi)linear equations of mathematical physics (regardless of their order), including classical elliptic, parabolic and hyperbolic equations, making it still attractive to the community. A nice historical exposition of the classical Friedrichs' theory (which was very active until 1970's) can be found in \cite{MJensen}. While we postpone the introduction of the precise definition of abstract Friedrichs operators to the subsequent section (Definition \ref{def:abstractFO}), here we discuss the main ideas. 
Assume we are given two densely defined linear operators $T_0$, $\widetilde{T}_0$ on a Hilbert space $\lH$ such that $T_0\subseteq \widetilde{T}_0^*$ and $\widetilde{T}_0\subseteq T_0^*$ ($T_0^*$ denotes the adjoint operator of $T_0$ and $\widetilde{T}_0\subseteq T_0^*$ is understood in the standard way: $\dom \widetilde{T}_0\subseteq\dom T_0^*$ and $T_0^*|_{\dom\widetilde{T}_0}=\widetilde{T}_0$). Now we seek for realisations (or extensions) $T$ of $T_0$, i.e.~$T_0\subseteq T\subseteq \widetilde{T}_0^*$, such that $T:\dom T\to\lH$ is bijective. Thus, if $T$ represents a differential operator (e.g.~a classical Friedrichs operator), the associated problem $Tu=f$, for $f\in\lH$, is well-posed, and the choice of realisation $T$  corresponds to the prescribed (initial-)boundary conditions. Since our focus is on further developing the abstract theory, we begin with a brief overview of the existing framework.

\begin{itemize}
    \item The authors introduced the concept of abstract Friedrichs operators in \cite{EGC} and obtained a well-posedness result within the operator-theoretic framework introduced for this purpose (see also \cite{ABcpde}). This is an immediate improvement over the classical theory, because no such satisfactory result was previously available.

    \item The renewed interest in Friedrichs systems arose from numerical analysis (see e.g.~\cite{HMSW, MJensen}) based on the need to apply (discontinuous) Galerkin finite element methods to partial differential equations of various types. The abstract approach initiated a number of new investigations in various directions. For example, studies of different representations of boundary conditions and the relation with classical theory \cite{ABcpde, ABjde, ABVisrn, AEM-2017, BH21, ES22}, applications to diverse (initial-)boun\-dary value problems of elliptic, hyperbolic, and parabolic type \cite{ABVjmaa, BEmjom, BVcpaa, EM19, EGsemel, MDS}, and the development of different numerical schemes \cite{BDG, BEF, CM21, CHWY23, EGbis, EGter}.

    \item In the classical setting, there were three different but equivalent ways to pose boundary conditions. One was Friedrichs' boundary condition formulated via matrix valued boundary fields \cite{KOF}, usually referred to as (FM)-boundary conditions. The second was the Friedrichs-Lax (FX)-boundary conditions, also called maximal boundary conditions \cite{FL}. The third approach, introduced in \cite{PS}, is referred to as the (FV)-boundary condition. In the abstract theory, there are analogous boundary conditions (M), (X) and (V) respectively. The (V)-boundary condition is also referred to as the cone-formalism, which was introduced in \cite{EGC}. A discussion of equivalence of these three types of boundary conditions can be found in \cite{EGC, ABcpde}. We shall recall the definitions in the next section.

    \item In \cite{AEM-2017} authors prove the existence and multiplicity of subspaces satisfying (V)-boundary conditions. It was also established that Grubb's universal extension theory \cite{Grubb} is applicable, leading to a classification result.

    \item A von-Neumann type decomposition of the graph space of abstract Friedrichs operators was proved in \cite{ES22}. In \cite{ES25}, it was recognised that abstract Friedrichs operators can be written as the sum of skew-symmetric and a bounded, strictly positive definite self-adjoint operator. In the same paper the von-Neumann extension theory for abstract Friedrichs operators was studied.
    
\end{itemize}

The first main result of this paper is to prove the equivalence between the $m-$accretive extensions and the (V)-boundary condition of abstract Friedrichs operators.
To place this result in context, 
let us briefly recall that the notions of accretive and $m-$accretive operators stem from the study of linear operators in Hilbert spaces and their applications in differential equations and stability analysis. 
Early work in functional analysis by researchers such as E. Hille and R. S. Phillips in the 1950s, who developed the theory of semigroups of linear operators, laid the groundwork for understanding accretive operators as key tools in the analysis of time-evolution problems. Over time, mathematicians such as H. Br\'ezis and G. Minty expanded these concepts to nonlinear operators, introducing the notion of maximal accretivity to ensure well-posedness in various contexts.

Returning to abstract Friedrichs operators, in \cite{BEmjom}, the authors proved that any realisation satisfying the (V)-boundary condition gives rise to a contractive $C_0-$semigroup. A straightforward application of the Lumer-Phillips theorem then shows that any such realisation is $m-$accretive. The main novelty of our result lies in the converse: namely, that any $m-$accretive realisation of abstract Friedrichs operators is precisely the realisation satisfying the (V)-boundary condition. In addition, we also prove that the skew-symmetric part of abstract Friedrichs operators has $m-$accretive extensions on the same domains. The classification of $m-$accretive extensions of skew-symmetric operators is well-studied (see \cite{ACE23,PT24, Trostorff23, WW20}). Thus, we can also classify all $m-$accretive realisations (or, equivalently, realisations satisfying the (V)-boundary condition) of abstract Friedrichs operators by applying the existing theory for skew-symmetric operators.

In the next section, we begin by fixing all definitions and notations. 
The result discussed above regarding the connection between the 
(V)-boundary condition and $m-$accretive realisations is presented in
Section \ref{sec:Overview}. 
The fourth section offers an interpretation of the (M), (X) and (V) boundary conditions in terms of $m-$accretive realisations. 
Finally, the paper concludes with some improved results on the classification of 
(M)-boundary conditions in Section \ref{sec:M=V}, illustrated with a second 
order PDE example.

\section{Abstract Friedrichs operators}\label{sec:abstractFO}

\subsection{Definition and main properties}

The abstract Hilbert space formalism for Friedrichs systems which we 
study in this paper was introduced 
and developed in \cite{EGC, ABcpde} for real vector spaces, while 
the required differences for complex vector spaces have been 
supplemented more recently in \cite{ABCE}. 
Here we present the definition in the form given in
\cite[Definition 1]{AEM-2017}.

\begin{definition}\label{def:abstractFO}
A (densely defined) linear operator $T_0$ on a complex Hilbert space $\lH$
(a scalar product is denoted by
$\scp\cdot\cdot$, which we take to be anti-linear in the second entry)
is called an \emph{abstract Friedrichs operator} if there exists a
(densely defined) linear operator $\widetilde{T}_0$ on $\lH$ with the following properties:
\begin{itemize}
 \item[(T1)] $T_0$ and $\widetilde{T}_0$ have a common domain $\lD$, i.e.~$\dom T_0=\dom\widetilde{T}_0=\lD$,
 which is dense in $\lH$, satisfying
 \[
 \scp{T_0\phi}\psi \;=\; \scp\phi{\widetilde T_0\psi} \;, \qquad \phi,\psi\in\mathcal{D} \,;
 \]
 \item[(T2)] there is a constant $\lambda>0$ for which
 \[
 \|(T_0+\widetilde{T}_0)\phi\| \;\leqslant\; 2\lambda\|\phi\| \;, \qquad \phi\in\mathcal{D} \,;
 \]
 \item[(T3)] there exists a constant $\mu>0$ such that
 \[
 \scp{(T_0+\widetilde{T}_0)\phi}\phi \;\geqslant\; 2\mu \|\phi\|^2 \;, \qquad \phi\in\mathcal{D} \,.
 \]
\end{itemize}
The pair $(T_0,\widetilde{T}_0)$ is referred to as a \emph{joint pair of abstract Friedrichs operators}
(the definition is indeed symmetric in $T_0$ and $\widetilde{T}_0$).
\end{definition}

\begin{remark}
Any pair of operators satisfying condition (T1), and consequently any pair of abstract Friedrichs operators, can be viewed as a particular case of \emph{adjoint}, \emph{dual} or \emph{symmetric} pairs. In the general framework, however, it is not required that those two operators possess identical domains (see e.g.~\cite{Beh25, JP17} and \cite[Chapter 13]{Grubb}).    
\end{remark}

Before moving to the main topic of the paper, 
let us briefly recall the essential properties 
of (joint pairs of) abstract Friedrichs operators, 
which we summarise in the form of a theorem.
At the same time, we introduce the notation that is used throughout the paper.
The presentation consists of two steps: first we deal with the consequences 
of conditions (T1)--(T2), and then we highlight the additional structure implied by condition (T3). 
A similar approach can be found in \cite[Theorem 2.2]{BEW23}.
\begin{theorem}\label{thm:abstractFO-prop}
Let a pair of linear operators $(T_0,\widetilde{T}_0)$ on $\lH$ satisfy {\rm (T1)} and {\rm (T2)}. Then the following holds.
\begin{enumerate}
\item[\rm{(i)}] $T_0\subseteq \widetilde{T}_0^*=:T_1$ and $\widetilde{T}_0\subseteq T_0^*=:\widetilde{T}_1$, where 
$\widetilde{T}_0^*$ and $T_0^*$ are adjoints of $\widetilde{T}_0$ and $T_0$, respectively.

\item[\rm{(ii)}] The pair of closures $(\overline{T}_0,\overline{\widetilde{T}}_0)$ satisfies {\rm (T1)--(T2)} with the same constant $\lambda$.

\item[\rm{(iii)}] $\dom \overline{T}_0=\dom\overline{\widetilde{T}}_0=:\lW_0$ and $\dom T_1=\dom\widetilde{T}_1=:\lW$.

\item[\rm{(iv)}] The graph norms $\|\cdot\|_{T_1}:=\|\cdot\|+\|T_1\cdot\|$ and $\|\cdot\|_{\widetilde T_1}
	:=\|\cdot\|+\|\widetilde T_1\cdot\|$ are equivalent, $(\lW,\|\,\cdot\,\|_{T_1})$ is a Hilbert space (the \emph{graph space})
	and $\lW_0$ is a closed subspace in it containing $\lD$.
	
\item[\rm{(v)}] The linear operator $\overline{T_0+\widetilde{T}_0}$ is everywhere defined, bounded and self-adjoint on 
$\lH$ that coincides on $\lW$ with $T_1+\widetilde{T}_1$.

\item[\rm{(vi)}] The expression 
\begin{equation}\label{eq:D1}
\dup{\lW'}{Du}{v}{\lW} \;:=\; \scp{T_1u}{v} 
- \scp{u}{\widetilde{T}_1v} \;,
\quad u,v\in\lW \,, 
\end{equation}
defines a bounded linear operator $D\in \mathcal{L}(\lW;\lW')$ that is called the \emph{boundary operator}, as $\ker D = \lW_0$.
The \emph{boundary form}
\begin{equation}\label{eq:D}
\iscp uv :=\dup{\lW'}{Du}{v}{\lW}  \;,
\quad u,v\in\lW \,, 
\end{equation}
defines an indefinite inner product on $\lW$ (cf.~\cite{Bo}) and we have $\lW^{[\perp]}=\lW_0$ and $\lW_0^{[\perp]}=\lW$, where
the $\iscp\cdot\cdot$-orthogonal complement of a set $X\subseteq \lW$
is defined by
\begin{equation*} 
X^{[\perp]} := \bigl\{u\in \lW : (\forall v\in X) \quad 
\iscp uv = 0\bigr\}
\end{equation*}
and it is closed in $\lW$. Moreover, $X^{[\perp][\perp]}=X$ if and only if 
$X$ is closed in $\lW$ and $\lW_0\subseteq X$. 

For future reference, let us define
\begin{equation}\label{eq:lWposneg}
\begin{aligned}
	\lW^+ &:= \bigl\{u\in\lW : \iscp{u}{u}\geq 0\bigr\} \\
	\lW^- &:= \bigl\{u\in\lW : \iscp{u}{u}\leq 0\bigr\} \,.
\end{aligned}
\end{equation}
Note that $X\subseteq X^{[\perp]}$ implies $X\subseteq \lW^+\cap\lW^-$.
\end{enumerate}

Assume, in addition, {\rm (T3)}, i.e.~$(T_0,\widetilde{T}_0)$ is 
a joint pair of abstract Friedrichs operators. Then
\begin{itemize}
\item[\rm{(vii)}] $(\overline{T}_0,\overline{\widetilde{T}}_0)$ satisfies {\rm (T3)} with the same constant $\mu$.
\item[\rm{(viii)}] A lower bound for $\overline{T_0+\widetilde{T}_0}$ is $2\mu>0$.
\item[\rm{(ix)}] We have
\begin{equation}\label{eq:decomposition}
\lW \;=\; \lW_0 \dotplus \ker T_1 \dotplus \ker\widetilde T_1 \;,
\end{equation}
where the sums are direct, $\lW_0 \dotplus \ker T_1 \subseteq \lW^-$, $\lW_0  \dotplus \ker\widetilde T_1  \subseteq \lW^+$ and all spaces on the right-hand side 
are pairwise $\iscp{\cdot}{\cdot}$-orthogonal. 
Moreover, the linear projections
\begin{equation}\label{eq:projections}
p_\mathrm{k} : \lW \to \ker T_1 \quad \hbox{and} \quad
p_\mathrm{\tilde k}:\lW\to \ker \widetilde{T}_1
\end{equation}
are continuous as maps $(\lW,\|\cdot\|_{T_1})\to (\lH,\|\cdot\|)$, i.e.~$p_\mathrm{k}, p_\mathrm{\tilde k}\in\lL(\lW,\lH)$.
\item[\rm{(x)}] Let $\lV$ be a subspace of the graph space $\lW$ such that
$\lW_0\subseteq\lV\subseteq \lW^+$ (see \eqref{eq:lWposneg}). Then 
$$
(\forall u\in \lV) \qquad \|T_1u\|\geq \mu\|u\| \,.
$$ 
In particular, $\overline{\ran (T_1|_\lV)}=\ran \overline{T_1|_\lV}$.

Analogously, if $\widetilde{\lV}$ is a subspace of $\lW$ such that 
$\lW_0\subseteq\widetilde{\lV}\subseteq\lW^-$, then 
$\|\widetilde T_1 v\|\geq \mu\|v\|$, $v\in\widetilde{\lV}$, 
and $\overline{\ran (\widetilde{T}_1|_{\widetilde{\lV}})}=
\ran \overline{\widetilde{T}_1|_{\widetilde \lV}}$.

\item[\rm{(xi)}] Let $\lV\subseteq\lW$ be a closed subspace (in $\lW$) containing $\lW_0$. 
Then, for a subspace $\widetilde{\lV}$ of $\lW$, the operators $T_1|_\lV$ and $\widetilde{T}_1|_{\widetilde{\lV}}$ are 
mutually adjoint, i.e.~$(T_1|_\lV)^*=\widetilde{T}_1|_{\widetilde{\lV}}$
and $(\widetilde{T}_1|_{\widetilde{\lV}})^*=T_1|_\lV$, if and only if
$\widetilde{\lV}=\lV^{[\perp]}$.
\item[\rm{(xii)}]  Let $\lV\subseteq\lW$ be a closed subspace containing $\lW_0$
such that $\lV\subseteq \lW^+$ and $\lV^{[\perp]}\subseteq\lW^-$.
Then $T_1|_\lV:\lV\to\lH$ and $\widetilde{T}_1|_{{\lV}^{[\perp]}}:\lV^{[\perp]}\to\lH$ are bijective,
i.e.~isomorphisms when we equip their domains with the 
graph topology, and for every $u\in\lV$ the following estimate holds:
\begin{equation}\label{eq:apriori}
	\|u\|_{T_1} \leq \Bigl(1+\frac{1}{\mu}\Bigr) \|T_1 u\| \,.
\end{equation}
The same estimate holds for $\widetilde{T}_1$ and ${\lV}^{[\perp]}$ replacing $T_1$ and $\lV$, respectively.

These bijective realisations of $T_0$ and $\widetilde{T}_0$ we call 
\emph{bijective realisations with signed boundary map}. 
\item[\rm{(xiii)}] Let $\lV\subseteq\lW$ be a closed subspace containing $\lW_0$.
Then $T_1|_\lV:\lV\to\lH$ is bijective if and only if $\lV\dotplus\ker T_1 = \lW$.
\end{itemize}
\end{theorem}

The statements i)--iv), vii) and viii) follow easily from the corresponding
assumptions (cf.~\cite{AEM-2017, EGC}).
The claims v), x) and xii) are already 
argued in the first paper on abstract Friedrichs operators \cite{EGC} for real vector spaces
(see sections 2 and 3 there), while in \cite{ABCE} 
the arguments are repeated in the complex setting.
The same applies for vi) with a remark that for a further structure 
of indefinite inner product space $(\lW, \iscp{\cdot}{\cdot})$ we refer to 
\cite{ABcpde}. 
The decomposition given in ix) is derived in \cite[Theorem 3.1]{ES22},
while for additional claims on projectors we refer to the proof of Lemma 3.5 in the 
aforementioned reference. In the same reference one can find the proof 
of part xiii) (Lemma 3.10 there).
Finally, a characterisation of mutual self-adjointness, xi), is obtained 
in \cite[Theorem 9]{AEM-2017}.

\begin{remark}
In this paper, we use the term \emph{realisation} to denote 
operators $T$ that lie between the minimal (e.g.~$T_0$) and the maximal
(e.g.~$T_1$) operators, i.e.~$T_0\subseteq T\subseteq T_1$. 
Moreover, in our context, it coincides with the concept of \emph{extensions}
of the minimal operator, since all studied operators in this manuscript 
are always between a minimal and a maximal operator
(which is not the case in all works, cf.~\cite[Proposition 3.3]{PT24}). Note that all realisations are densely defined.
\end{remark}

\subsection{Boundary conditions for Friedrichs operators}
We recall the definitions and connections among different boundary conditions for abstract Friedrichs operators, which are mainly followed from the references \cite{ABcpde, EGC}, while a brief overview can also be found in \cite[Chapter 2.5]{Soni24}. We assume that $(T_0,\widetilde T_0)$ is a joint pair of abstract Friedrichs operators on a Hilbert space $\lH$.

\begin{definition}[(V)-boundary conditions]\label{dfn:V1-V2}
	 A subspace $\lV$ of the graph space $\lW$ is said to satisfy (V)-\emph{boundary conditions} (jointly with $\widetilde \lV := \lV^{[\perp]}$) if the following conditions are satisfied:
	\begin{itemize}
		\item[(V1)] The boundary form has opposite signs on these spaces. More precisely, $\lV\subseteq \lW^+$ and $\widetilde\lV\subseteq \lW^-$, i.e.
		\begin{align*}
		& (\forall u\in \lV)\qquad \iscp{u}{u}\;\geq\; 0\,,\\
		& (\forall v\in \widetilde \lV) \qquad \iscp{v}{v}\;\leq\; 0\;.
		\end{align*}
		
		\item[(V2)] The subspaces $\lV, \widetilde \lV$ are mutually $\iscp{\cdot}{\cdot}$-orthogonal, i.e.
		\begin{align*}
		\lV \;=\; \widetilde \lV^{[\perp]} \quad \mathrm{and} \quad \widetilde \lV \;=\; \lV^{[\perp]}\;.
		\end{align*}
	\end{itemize}
\end{definition}
Note that the second equality in (V2) is just the definition of $\widetilde{\lV}$, and it is repeated here to emphasise symmetry in conditions for both subspaces..
\begin{remark}\label{rem:Vcond}
	Let us note that any subspace $\lV$ of $\lW$ that satisfies assumption {\rm (V)} (in pair with $\widetilde \lV := \lV^{[\perp]}$) is, by Theorem \ref{thm:abstractFO-prop}.vi),  closed and contains $\lW_0$, and thus satisfies assumptions from part xii) of the same theorem. Since the converse is also true (see again Theorem \ref{thm:abstractFO-prop}.vi)) it follows that (V)-boundary conditions correspond to realisations with signed boundary map.
	If this is the case, note also that by part xiii) of the aforementioned theorem we have $\lV\dotplus\ker T_1 = \lW$. 
    However, not all bijections described in that part are with signed boundary map (see Example \ref{ex:ex-semigroup} below). 
\end{remark}

\begin{definition}[(X)-boundary conditions]
	A subspace $\lV$ of $(\lW,\iscp{\cdot}{\cdot})$ is said to satisfy (X)-\emph{boundary conditions} if it is maximal non-negative, i.e.~if the following conditions hold:
	\begin{itemize}
		\item[(X1)] $\lV$ is non-negative with respect to the boundary form $\iscp{\cdot}{\cdot}$, i.e.~$\lV\subseteq \lW^+$.
		
		\item[(X2)] There is no non-negative subspace of $(\lW,\iscp{\cdot}{\cdot})$ containing $\lV$ properly.
	\end{itemize}
	
\end{definition}
In a similar way one can define the notion of maximal non-positive subspace of $\lW$. Actually, it can be proven that a subspace $\lV$ of $\lW$ is maximal non-negative if and only if $\widetilde \lV := \lV^{[\perp]}$ is maximal non-positive in $\lW$ (see Theorem \ref{thm:v=x} below).

\begin{definition}[(M)-boundary conditions]\label{dfn:M-boundary} Let $D$ be the boundary operator. An operator $M\in \mathcal{L}(\lW;\lW')$ is said to satisfy (M)-boundary conditions if:
	\begin{itemize}
		\item[(M1)] $M$ is non-negative: 
		\begin{align*}
		(\forall u\in \lW)\qquad \Re \dup{\lW'}{Mu}{u}{\lW}\geq 0\,,
		\end{align*}
		where $\Re z$ stands for the real part of a complex number $z$;
		\item[(M2)] the graph space can be decomposed as
		\begin{align*}
		\lW=\ker(D-M)+\ker(D+M)\;.
		\end{align*}
	\end{itemize}

\end{definition}

The following result \cite[Lemma 4.1]{EGC} justifies the usage of the notion boundary operator for $M$, as well.

\begin{lemma}\label{EGC:lem4.1} Let $M\in \mathcal{L}(\lW;\lW')$ satisfies {\rm(M)}-boundary conditions. Then,
	\begin{align*}
	\ker D\;=\;\ker M\;=\;\ker M^*\ \quad \mathrm{and}\quad \ran D\;=\;\ran M\;=\;\ran M^*\;. 
	\end{align*}
\end{lemma}
The statement of equivalence between (X) and (V) boundary conditions is fairly straightforward and can be summarised as follows (cf.~\cite[Theorem 2]{ABcpde}).

\begin{theorem}\label{thm:v=x}
	Let $\lV$ be a subspace of $\lW$.
	\begin{itemize}
		\item[\rm{(a)}] If $\lV$ satisfies {\rm{(V)}}-boundary conditions, then $\lV$ is maximal non-negative in $\lW$ (and $\widetilde \lV := \lV^{[\perp]}$ is maximal non-positive).
		\item[\rm{(b)}] If $\lV$ satisfies {\rm{(X)}}-boundary conditions, then $\lV$ satisfies {\rm{(V)}}-boundary conditions.
	\end{itemize}
\end{theorem}

The topic of equivalence between (V) and (M) boundary conditions appeared to be more challenging. Eventually the following theorem was proven.

\begin{theorem}\label{lem:m-implies-v}
Let $\lW$ be the graph-space and $D$ the boundary operator.
\begin{itemize}
	\item[\rm{(a)}] 	If $M\in \mathcal{L}(\lW;\lW')$ is an operator satisfying (M)-boundary conditions, then the subspace  $\lV:=\ker(D-M)$ satisfies \emph{(V)}-boundary conditions, and $\lV^{[\perp]}=\ker(D+M^*)$.
	\item[\rm{(b)}] If $\lV$ satisfies {\rm{(V)}}-boundary conditions, then there exists an operator $M\in \mathcal{L}(\lW;\lW')$ satisfying (M)-boundary conditions such that $\lV:=\ker(D-M)$ and $\lV^{[\perp]}=\ker(D+M^*)$.
\end{itemize}
\end{theorem}

The (a) part of the above theorem was proved in \cite[Theorem 4.2]{EGC}. The converse appeared to be more challenging. In some cases, this question boils down to closedness of the subspace $\lV+\widetilde \lV$ in the graph space $\lW$ (see e.g.~\cite[Section 4]{EGC}). In \cite[Corollary 3]{ABcpde}, this problem has been addressed in full generality. We shall give more details on this subject in the fifth section.

\subsection{The von Neumann extension theory} 

Recently, in \cite{ES25}, the authors presented a classification theory in the spirit of the von Neumann approach, which is well-known theory for symmetric as well as skew-symmetric operators. Here we briefly recall the results of the paper \cite{ES25} (see also \cite[Chapter 3]{Soni24}), in the context of the requirement of this manuscript.

\begin{theorem}\label{thm:abstractFO_L0+S}
A pair of densely defined operators $(T_0,\widetilde{T}_0)$ on $\lH$ 
is a pair of abstract Friedrichs operators if and only if there exist a densely defined skew-symmetric operator $L_0$ and a bounded self-adjoint operator $S$ with strictly positive bottom,
both on $\lH$, such that
\begin{equation}\label{eq:abstractFO_L0+S}
	T_0=L_0+S \qquad \hbox{and} \qquad \widetilde{T}_0=-L_0+S \,.
\end{equation}
If this is the case then $S=\frac{1}{2}\left( \overline{T_0 + \widetilde T_0}\right)$, and thus
\begin{equation}\label{eq:S-ineq}
\|Su\|\le\lambda \|u\|, \quad \scp{S u}{u} \geq \mu\|u\|^2 \;, \quad u\in\lH\,,
\end{equation}
where $\lambda,\mu>0$ are constant appearing in {\rm (T2)} and {\rm (T3)}, respectively.

For a given pair, the decomposition \eqref{eq:abstractFO_L0+S}
is unique. Furthermore, we have the following:
\begin{itemize}
    \item[\rm{(i)}] If we denote $L_1:=-L_0^*\supseteq L_0$, then we have
\begin{equation}\label{eq:TLSnotation}
\begin{split}
T_1 &= L_1 + S \;, \qquad
\end{split}
\begin{split}
\widetilde{T}_1 &=-L_1+S \;.
\end{split}
\end{equation}
\item[\rm{(ii)}]  $\lW_0=\dom \overline{L}_0$ and $\lW=\dom L_1$, 
i.e.~spaces $\lW_0$ and $\lW$ are independent of $S$. 

\item[\rm{(iii)}] The boundary form satisfies
\begin{equation}\label{eq:iscp_L1}
\iscp{u}{v} = \scp{L_1 u}{v}+ \scp{u}{L_1v} \;, \quad u,v\in\lW \,.
\end{equation}
\end{itemize}
\end{theorem}

\begin{corollary}\label{cor:signed_bij}
Let $(T_0,\widetilde{T}_0)$ be a joint pair of abstract Friedrichs operators
on $\lH$ and let $\lV\subseteq\lW$ be a closed subspace containing $\lW_0$
such that $\lV\subseteq\lW^+$ and $\lV^{[\perp]}\subseteq\lW^-$ (with respect to $(T_0,\widetilde{T}_0)$).
For any joint pair of abstract Friedrichs operators $(A_0,\widetilde{A}_0)$ 
on $\lH$ such that 
$$
(A_0-\widetilde{A}_0)^*=(T_0-\widetilde{T}_0)^*
$$ 
we have that $\bigl((\widetilde{A}_0)^*|_{\lV},(A_0)^*|_{\lV^{[\perp]}}\bigr)$
is a pair of bijective realisations with signed boundary map.
\end{corollary}

\begin{theorem}\label{thm:classif}
Let $(T_0,\widetilde{T}_0)$ be a joint pair of abstract Friedrichs operators 
on $\lH$ and let $T$ be a closed realisation of $T_0$, i.e.~$T_0\subseteq T\subseteq T_1$. For a mapping $U:(\ker\widetilde{T}_1,\iscp{\cdot}{\cdot}) \to (\ker T_1,-\iscp{\cdot}{\cdot})$ we define $\lV_U:=\bigl\{u_0+U\tilde\nu + \tilde\nu : u_0\in\lW_0, \, \tilde\nu\in\ker \widetilde{T}_1\bigr\}$.
\begin{itemize}
\item[\rm{(i)}] $T$ is bijective if and only if there exists a bounded linear operator
$U:\ker\widetilde{T}_1\to\ker T_1$ such that $\dom T=\lV_U$.

\item[\rm{(ii)}]  $T$ is a bijective realisation with signed boundary map 
if and only if there exists a linear operator $U:(\ker\widetilde{T}_1,\iscp{\cdot}{\cdot}) \to (\ker T_1,-\iscp{\cdot}{\cdot})$ such that $\|U\|\leq 1$ and $\dom T=\lV_U$.

\item[\rm{(iii)}] $\dom T= \dom T^*$ if and only if there exists a unitary transformation
$U:(\ker\widetilde{T}_1,\iscp{\cdot}{\cdot}) \to (\ker T_1,-\iscp{\cdot}{\cdot})$ such that $\dom T=\lV_U$.

\item[\rm{(iv)}] The mapping $U\mapsto T_1|_{\lV_U}$, is a one-to-one correspondence between the classifying operators $U$ and the realisations $T$, i.e.~$\dom T$, in each of the above cases.
\end{itemize}
\end{theorem}

\section{\texorpdfstring{$m$-accretive}{m-accretive} extensions of abstract Friedrichs operators}\label{sec:Overview}
In this section we shall further explore the decomposition from Theorem \ref{thm:abstractFO_L0+S} in the context of accretive operators. Since, for historical reasons, the use of term \emph{accretive} is not uniform in the literature, we shall provide a brief overview of basic notions that we shall use. 

A linear operator $A: \dom A \subseteq \lH \to \lH$ on a complex Hilbert space $\lH$ is said to be \emph{accretive} if for all $x \in \dom A$, we have
\begin{equation*}
\Re \scp{Ax}{x} \geq 0.
\end{equation*}
    An accretive operator is said to be \emph{maximal accretive} if it has no proper accretive extension, while an accretive operator $A: \dom A \subseteq \lH \to \lH$ is called \emph{$m$-accretive} if $I+A$ is onto (surjective). An operator $A$ is referred to as (\emph{maximal}) \emph{dissipative} if $-A$ is  (maximal) accretive. 
    We consider only densely defined operators on a Hilbert space in this paper. For these operators, the notions of maximal accretivity and m-accretivity are equivalent \cite[p.~201]{P}.

In the context of linear differential equations, maximal accretive operators generate $C_0-$ semigroups of contractions (the Lumer-Phillips theorem - see \cite{Engel2, Pazy}) that provide solutions to time-evolution problems. The notion of $m$-accretivity is also central to the Hille-Yosida theorem \cite{Engel2, Pazy}, through the concept of the so-called range condition: an accretive operator is $m-$accretive if and only if the range of $\lambda I + A$ equals $\lH$ for some (and thus all) $\lambda >0$. In the context of nonlinear operators this is closely related the Minty-Browder theorem \cite{minty1962monotone, zeidler1985nonlinear}, a result initially proven for monotone operators.


In the following results we assume that $(T_0, \widetilde T_0)$ is a joint pair of abstract Friedrichs operators on $\lH$ and make use of decomposition $T_i=L_i + S$ from Theorem \ref{thm:abstractFO_L0+S}, with $i$ being $0, 1$ or void.
\begin{lemma}\label{lem:accretive}
    If $T$ is a realisation of $T_0$, i.e.~$T_0\subseteq T\subseteq T_1$, then  the following statements are equivalent:
   \begin{enumerate}
   	\item[\rm{(i)}] $T$ is accretive;
   	\item[\rm{(ii)}] $L$ is accretive;
   	\item[\rm{(iii)}]  $\dom T \,\,(= \dom L)$ is non-negative (contained in $\lW^+$).
   \end{enumerate}
\end{lemma}
\begin{proof} 
Note that \eqref{eq:iscp_L1} implies
\begin{align}\label{eq:iscp-L}
\Re\scp{L_1 u}{u}=\frac{1}{2}\iscp{u}{u}\,,\quad u\in \lW\;,
\end{align}
which proves equivalence between (ii) and (iii), as $L=L_1|_{\dom T}$.
From $T=L+S$  and the second inequality in (\ref{eq:S-ineq}) it is clear that accretivity of $L$ implies the accretivity of $T$.
It remains to prove that (i) implies (iii): from $T=L+S$ and (\ref{eq:iscp-L}), we have
\begin{equation}
 \begin{aligned}\label{eq:iscp-T-L}
\Re\scp{Tu}{u}=&\Re\scp{Lu}{u} + \scp{S u}{u}\\
=&\frac{1}{2}\iscp{u}{u} + \scp{S u}{u}\,, \quad u \in \dom T=\dom L\,.
\end{aligned}
\end{equation}
For an arbitrary $u\in \dom T$, by Theorem \ref{thm:abstractFO-prop}(ix), there exist $u_0\in \lW_0$, $\nu\in \ker T_1$ and $\tilde\nu\in \ker\widetilde{T}_1$ such that $u=u_0+\nu+\tilde\nu$. Since $\lD=\dom T_0\subseteq \dom T$, for any $v_0\in \lD$ we also have that $v=v_0+\nu+\tilde\nu\in \dom T$, and $\iscp{u}{u}=\iscp{\nu}{\nu}+\iscp{\tilde\nu}{\tilde\nu} = \iscp{v}{v}$ by $\iscp{\cdot}{\cdot}$-orthogonality from Theorem \ref{thm:abstractFO-prop}(ix). Thus, if $T$ is accretive, using this, (\ref{eq:iscp-L}), (\ref{eq:iscp-T-L}) and the first inequality in (\ref{eq:S-ineq}), we get  
    \begin{align*}
        \Re\scp{Lu}{u}=\frac{1}{2} \iscp{u}{u}=\frac{1}{2} \iscp{v}{v}= \Re\scp{Tv}{v} - \scp{Sv}{v}\geq -\lambda\|v\|^2 = -\lambda\|v_0+\nu+\tilde\nu\|^2\;.
    \end{align*}
    Since $\dom T_0$ is dense in $\lH$, we can choose $v_0$ arbitrarily close to $-\nu-\tilde\nu$ with respect to norm $\|\cdot\|$. Hence,
    \begin{align*}
        \Re\scp{Lu}{u}\geq 0\,,
    \end{align*}
    implying that $L$ is accretive.
    
\end{proof}

\begin{theorem}\label{thm:diss-02}
   If $T$ is a realisation of $T_0$, then the following assertions are equivalent:
    \begin{enumerate}
        \item[\rm{(i)}] $T$ is $m-$accretive;
        \item[\rm{(ii)}] $L$ is $m-$accretive;
        \item[\rm{(iii)}] $T$ is a bijective realisation with signed boundary map, i.e.~$\dom T$ satisfies {\rm(V)}- (equivalently {\rm(X)}-) boundary conditions.
    \end{enumerate}
\end{theorem}

\begin{proof}
	 For the equivalence between (i) and (ii), from Lemma \ref{lem:accretive} we have that $T$ is accretive if and only if $L$ is accretive.  Since $\dom T=\dom L$ by definition, it is easy to see that if one of them is maximal accretive then so is the other, which concludes this equivalence.
	 
    Let us now prove that (i) implies (iii): if $T$ (and thus also L) is $m-$accretive, then from Lemma \ref{lem:accretive} we have $\dom T=\dom L\subseteq \lW^+$. By Theorem \ref{thm:v=x} and Theorem \ref{thm:abstractFO-prop}(xii) it is enough to prove that $\dom L$ is maximal non-negative. If this is not the case then there exists a nonnegative subspace $\lV \subseteq \lW^+$ that properly contains $\dom L$.
    But then $L_1|_{\lV}$ is again accretive by Lemma \ref{lem:accretive}, which contradicts $m-$accretivity of $L$.
    
    For the converse, if $T$ is a bijective realisation with signed boundary map then for any $u\in \dom T$, it holds $\iscp{u}{u}\geq 0$, which, by Lemma \ref{lem:accretive}, implies that $T$ is accretive. Since, by Theorem \ref{thm:v=x}, $\dom T$ is maximal non-negative in $\lW$, due to Lemma \ref{lem:accretive} it follows that $T$ is maximal accretive.   
    
\end{proof}

\begin{remark}
	Theorem \ref{thm:diss-02} provides an equivalence between (V)-boundary conditions and maximal accretive realisations of abstract Friedrichs operators. In the context of equivalence between (V)- and (X)-boundary conditions (see Theorem \ref{thm:v=x}) we can interpret this equivalence as follows: $T$ is a maximal accretive realisation of an abstract Friedrichs operator if and only if $\dom T$ is maximal nonnegative subspace of $\lW$. We shall explore these connections more precisely in the next section.
\end{remark}

\begin{remark}
    Recall that (by Theorem \ref{thm:classif}), point (iii) of the previous Theorem is equivalent to the fact that there exists a linear operator $U:(\ker\widetilde{T}_1,\iscp{\cdot}{\cdot})\to (\ker T_1, -\iscp{\cdot}{\cdot})$ such that $\|U\|\leq 1$ and $\dom T = \lW_0\dotplus \{\tilde\nu + U\tilde\nu:\tilde\nu\in\ker\widetilde{T}_1 \}$. Hence, all $m-$accertive realisations of $T_0$ can be parameterised by such contractive mappings $U$ between the kernels.
\end{remark}

Now we state the main result of this section which is now a direct consequence of the Lumer-Phillips theorem (see e.g.~\cite[Chapter II, Theorem 3.15]{Engel2}).
\begin{corollary}\label{cor:diss2}
     If $T$ is a realisation of $T_0$, then the following assertions are equivalent:
    \begin{itemize}
        \item[\rm{(i)}] $T$ is a bijective realisation with signed boundary map, i.e.~$\dom T$ satisfies {\rm(V)}- (equivalently {\rm(X)}-) boundary conditions;
        \item[\rm{(ii)}] $-T$ is a generator of a contractive $C_0-$semigroup;
        \item[\rm{(iii)}] $-L$ is a generator of a contractive $C_0-$semigroup.
    \end{itemize}
\end{corollary}

In \cite[Theorem 2]{BEmjom}, the implication (i)$\implies$(ii) has been proved.  Here in Corollary \ref{cor:diss2}, we show that the converse is also true, i.e.~the bijective realisations which give rise to the generators of contractive $C_0-$semigroups, are precisely the ones with signed boundary maps.

It is natural to ask whether weakening the assumptions on the realisation---such as considering all bijective realisations (not necessarily with a signed boundary map)---still ensures that the operator generates a strongly continuous semigroup, even if it is not contractive.
However, the answer is negative. More precisely, a bijective realisation without a signed boundary map may generate a (non-contractive) $C_0-$semigroup, but it may also fail to be a generator. This can already be observed in the following one-dimensional example.

\begin{example}\label{ex:ex-semigroup}
   Consider the first-order differential operator $L_0u=\displaystyle\frac{du}{dx}$ on $\lH=L^2((0,1);\R)$ (for simplicity, we consider only the real setting), with $\dom L_0= C_c^{\infty}((0,1))$. The minimal space and the graph space are $\lW_0=H^1_0((0,1))$ and $\lW = H^1((0,1))$, respectively, while the boundary operator $D$ and the boundary form $\iscp{\cdot}{\cdot}$ are given by
    \begin{align*}
        (\forall u,v\in \lW) \qquad 
            \dup{\lW'}{Du}{v}{\lW}=\iscp{u}{v} 
            = u(1)v(1)-u(0)v(0)\;.
    \end{align*}
      Note that an evaluation of a function from $H^1((0,1))$ at a point has a meaning due to embedding of $H^1((0,1))$ in the space of continuous functions. The operator $ T_0:=L_0+\mathbbm{1}$ is an abstract Friedrichs operator by Theorem \ref{thm:abstractFO_L0+S} (see also \cite[Example 1]{ABmn}) and,
      for $\alpha\in\R\cup\{\infty\}$, let us consider closed realisations $T^\alpha$, $T_0\subseteq T^\alpha\subseteq T_1$ (here we continue to use the notation introduced in Section \ref{sec:abstractFO}), where
    \begin{align*}
        \dom T^\alpha &= \bigl\{u\in\lW : u(1)=\alpha u(0)\bigr\} \;, \quad \alpha\in\R\,, \\
        \dom T^\infty &= \bigl\{u\in\lW : u(0)=0\bigr\} \;.
    \end{align*}
    Of course, we have $T^\alpha=L^\alpha+\mathbbm{1} = \displaystyle\frac{d}{dx}+\mathbbm{1}$, where 
    $L^\alpha=(-L_0)^*|_{\dom T^\alpha}$, and the derivative is taken in the weak sense.
    
    By \cite[Example 3.6]{ES25}, we know that all $T^\alpha$
    except for one particular value, $\alpha=e^{-1}\in (0,1)$, are bijective realisations. 
    Moreover, it is also shown that bijective realisations with 
    signed boundary map correspond to $\alpha\not\in (-1,1)$.

    Therefore, by Theorem \ref{thm:diss-02}, we know that, both
    $T^\alpha$ and $L^\alpha$ are $m-$accretive if and only if $\alpha\in\bigl(\R\cup\{\infty\}\bigr)\setminus (-1,1)$. 
    
    Our aim is to investigate whether $T^\alpha$ and $L^\alpha$
    generate a $C_0-$semigroup in the yet undetermined case when $\alpha\in (-1,1)$.
    Since $T^\alpha$ is merely a bounded perturbation of $L^\alpha$, and vice versa, it suffices to study 
    only $L^\alpha$ (cf. \cite[Chapter III, Theorem 1.3]{Engel2}). 
    Using this argument we can also see that the value 
    $\alpha=e^{-1}$ should not be of a particular importance 
    since it depends on the choice of bounded perturbation 
    (see \cite[Example 3.6]{ES25}),
    while for the property of being a generator of 
    $C_0-$semigroup this choice is irrelevant. 

    If $L^\alpha$ generates a $C_0-$semigroup $S^\alpha(t)$,
    then for any $u_0\in \dom L^\alpha$ the abstract Cauchy 
    problem 
    \begin{align*}
        u' = L^\alpha u \;, \quad   u(0) = u_0 \;,
    \end{align*}
    has a unique solution $S(t)u_0$. 
    Since the above differential equation corresponds to 
    a first-order transport partial differential equation,
    it is easy to see that the only choice for $S(t)$ is 
    to be a translation semigroup
    \begin{equation*}
        \bigl(S(t)u_0\bigr)(x) = u_0(x-t)
    \end{equation*}
    (cf.~\cite[Chapter I, Section 3]{Engel2}). In order to
    give a meaning to the formula above, we need to extend 
    $u_0$ to the negative part of the real line. 
    This is done by utilising the property $S(t)u_0\in \dom 
    L^\alpha$ (see e.g.~\cite[Chapter II, Lemma 1.3]{Engel2}).
    Hence, we get
    $$
    \alpha u_0(-t) = u_0(1-t) \,, \quad t\geq 0\,.
    $$
    
    For $\alpha=0$ the above property implies that 
    $u_0\equiv 0$, implying that $L^0$ cannot be a 
    generator (the abstract problem has a solution only for 
    $u_0\equiv 0$). 

    On the other hand, for $\alpha\in (-1,1)\setminus\{0\}$
    we get that $S(t)u_0$ is well-defined for any $u_0\in\dom L^\alpha$. In particular, $L^\alpha$ is a generator 
    of a $C_0-$semigroup. Moreover, one easily gets
    $$
    \bigl(S(n)u_0\bigr)(x) = \alpha^{-n} u_0(x) \,, \quad 
        x\in [0,1], \, n\in\N \,.
    $$
    Since $0\not=|\alpha|<1$, this provides evidence that
    these semigroups are neither contractive nor uniformly 
    bounded. 

    To conclude, in this example one can see that there is a 
    realisation of $L_0$, which corresponds to a bijective realisation of $T_0$, that does not generate a $C_0-$semigroup, while there are also realisations that 
    generate $C_0-$semigroups that are not uniformly bounded.

    Let us strengthen the above argument slightly by
    demonstrating in an alternative way that the operator 
    $L^0$ (i.e. for $\alpha=0$) does not generate a $C_0-$semigroup: by direct computation, the resolvent operator $\mR(\lambda, L^0)=(\lambda\mathbbm{1}-L^0)^{-1}$ 
    can be characterised as follows. For any $\lambda>0$ and $f\in L^2((0,1))$,
    we have
\begin{align*}
    \mR(\lambda, L^0)f(x) =  \int_x^1 e^{\lambda (x-y)}f(y) dy\;.
\end{align*}
For constant function $f\equiv 1$, the above reads
\begin{align*}
    \mR(\lambda, L^0)f(x) = \frac{1}{\lambda} - \frac{1}{\lambda e^\lambda} e^{\lambda x}\;.
\end{align*}
Thus, we have
\begin{align*}
    \|\mR(\lambda,L^0)\|^2\geq \|\mR(\lambda, L^0) f\|^2 = \frac{3}{2\lambda}\Bigl| 1- \frac{2}{3\lambda}-\frac{4}{3 e^\lambda} + \frac{1}{3 e^{2\lambda}}\Bigr| \,,
\end{align*}
implying that $\|\mR(\lambda,L^0)\|$ is of the order
$O(\frac{1}{\sqrt{\lambda}})$ for large $\lambda>0$.
Therefore, the assumption of the Hille-Yosida theorem 
\cite[Chapter I, Theorem 5.3]{Pazy} are not fulfilled, 
and thus $L^0$ (and consequently $T^0$) is not 
a generator of a $C_0-$semigroup.    
\end{example}

 \begin{remark}\label{rem:tilde-T}
 	   Analogous results to those of Lemma \ref{lem:accretive}, Theorem \ref{thm:diss-02} and Corollary \ref{cor:diss2} hold for realisations of $\widetilde{T}_0$, due to the symmetric role of $T_0$ and $\widetilde{T}_0$ in the theory of Friedrichs systems. For example, statements of Lemma \ref{lem:accretive} hold true with $T_i$ replaced with $\widetilde{T}_i$ (with $i$ being $0, 1$ or void), $L$ replaced with corresponding $\widetilde{L}:=-L_1|_{\dom \widetilde{T}}$ and $\lW^+$ with $\lW^-$. In particular,  $T_1|{_\lV}$ is $m-$accretive if and only if $\widetilde{T}_1|_{\lV^{[\perp]}}$ is $m-$accretive. This can also be seen from the semigroup theory of adjoint operators (\cite[Chapter I, Proposition 1.13]{Engel2}) and the fact that bijective realisations of $T_0$ and $\widetilde{T}_0$ with signed boundary map are mutually adjoint to each other (see Theorem \ref{thm:abstractFO-prop}.xi)). 
\end{remark}  

\begin{remark}
    A specific case of (V)-boundary condition occurs when $\lV^{[\perp]} =\lV$, i.e.~when $\dom T = \dom T^*$. It is easy to verify that this is equivalent to operator $L$ being skew-selfadjoint, which, by \emph{Stone's Theorem} \cite[Chapter II, Theorem 3.24]{Engel2}, is equivalent to the fact that $L$ generates a unitary $C_0$-group. However, it is worth noting that the operator $-T$ does not generate a unitray group, even though $-T$ generates a contractive $C_0-$semigroup. Indeed, if it does, then by Stone's theorem $T$ has to be skew-selfadjoint, i.e.~$T^*=-T$, which would imply $S=0$ in \eqref{eq:abstractFO_L0+S}, contradicting the fact that $S$ is strictly positive. The reason is that the operator $-T=-L-S$ is not $m-$accretive, which means the semigroup generated by $T$ is not contractive, due to the Lumer-Phillips theorem \cite[Chapter II, Theorem 3.15]{Engel2}. However, the $C_0-$semigroup $((P(t))_{t\geq 0}$ generated by $T$ has the following bound (see \cite[Chapter III, Theorem 1.3]{Engel2})
\begin{align*}
    \|P(t)\|\leq e^{\|S\|t}\qquad t\geq 0\;.
\end{align*}

\end{remark}

Let us conclude this section with a remark on skew-symmetric operators.

\begin{remark}
    The extension theory for skew-symmetric operators is well-studied. The von Neumann approach can be found in \cite{PT24} and the approaches involving the concept of boundary systems (or boundary quadruples) can be found in \cite{ACE23,Trostorff23, WW20}. These extension results, classifying all $m-$accretive extensions of skew-symmetric operators, combined with Theorem \ref{thm:diss-02}, provide another way to classify the boundary conditions of abstract Friedrichs operators of interest. Moreover, \cite{Trostorff23, WW20} also include the nonlinear $m-$accretive extensions of skew-symmetric operators, which leads to the nonlinear $m-$accretive extensions of abstract Friedrichs operators. 
\end{remark}

\section{Interpretation of \texorpdfstring{$m$-accretive}{m-accretive} extensions in terms of other boundary conditions}

In the previous section we proved the equivalence between $m-$accretive extensions and the realisations with signed boundary maps for abstract Friedrichs operators. In this section, we investigate and compare (V) and (X)-boundary conditions with $m-$accretivity in more details.

\begin{remark}
\textbf{(V)-boundary conditions}  are also known as the \emph{cone formalism}. The two parts (V1) and (V2) reflect on accretivity and maximality respectively. Indeed, to verify this interpretation of (V1) condition, let $T$ and $\widetilde{T}$ be realisations (extensions) of $T_0$ and $\widetilde{T}_0$, respectively. If we denote $\lV:=\dom T$, $\widetilde{\lV} :=\dom\widetilde{T}$, then by Lemma \ref{lem:accretive} (see also Remark \ref{rem:tilde-T}) the following assertions are equivalent:
     \begin{itemize}
         \item[(i)] $\lV$ and $\widetilde{\lV}$ satisfy {\rm{(V1)}}-condition;
         \item[(ii)] $T$ and $\widetilde{T}$ are accretive.
     \end{itemize} 
Similarly, having that $T$ and $\widetilde{T}$ are accretive, the maximal accretivity (of both $T$ and $\widetilde{T}$) is equivalent to (V2) condition. Indeed, 
due to Theorem \ref{thm:abstractFO-prop}(xi), (V2)-condition is equivalent to the fact that the operators $(T,\widetilde{T})$ are mutually adjoint, i.e.~
\begin{align*}
    T^*=\widetilde{T} \quad \mathrm{and}\quad \widetilde{T}^*=T\;.
\end{align*}
Since both are accretive, then by \cite[Chapter I, Proposition 1.3]{Engel2}, both are maximal accretive. The converse statement has already been proven in Theorem \ref{thm:diss-02} (see also Remark \ref{rem:tilde-T}). 
\end{remark}

\begin{remark}
The maximal accretive extensions of Friedrichs operators are directly comparable to \textbf{(X)-boundary conditions} through Lemma \ref{lem:accretive} and Theorem \ref{thm:diss-02}. To be more specific, (X1) condition is equivalent to the accretivity, while the condition (X2), which is the maximality condition in the context of non-negative subspaces,  appears to be equivalent to the maximality in the context of accretive operators. In a way, this provides another equivalence between (X) and (V) boundary conditions via maximal accretivity. 
\end{remark}

\begin{remark}
In the theory of abstract Friedrichs operators, a key challenge is to establish the connection between (V) (or (X)) boundary conditions and the abstract version of Friedrichs' conditions, namely, the \textbf{(M)-boundary conditions}. Naturally, the closer we get to understanding these connections, the easier it becomes to relate maximal accretivity to (M)-boundary conditions. Although there are existing results on the equivalence between (M) and (V) boundary conditions (see Theorem \ref{lem:m-implies-v} and its constructive version in Theorem \ref{thm:ABcpde-thm08} below, taken from \cite{ABcpde}), in the next section we provide a refined result in this direction. 
The equivalence between maximal accretivity and (V)-boundary conditions allows us to directly relate maximal accretivity to (M)-boundary conditions. Indeed, using theorems \ref{lem:m-implies-v} and \ref{thm:diss-02}, we can easily conclude that if an operator $M:\lW\to \lW'$ satisfies (M)-boundary conditions, then $T_1|_{\ker(D-M)}$ and $\widetilde{T}_1|_{\ker(D+M^*)}$ are $m-$accretive. Conversely, if a realisation $T$ of $T_0$ in $m-$accretive, then there exists an operator $M\in \mathcal{L}(\lW;\lW')$ satisfying (M)-conditions such that $\dom T:=\ker(D-M)$ and $\dom T^{[\perp]}=\ker(D+M^*)$.

The difficult part is the construction of such an operator $M$ for a given $m-$accretive realisation $T$, and we shall address it in the next section.
\end{remark}

\section{On equivalence of  (V) and (M) boundary conditions}\label{sec:M=V}

In this section we provide a more concrete proof of the equivalence between (V) and (M) boundary condition for a pair of abstract Friedrichs operators. 

More precisely, the known result \cite{ABcpde} is the following:

\begin{theorem}\label{thm:ABcpde-thm08}
Let $\lW$ be the graph space and $D$ the boundary operator.
\begin{itemize}
  \item[(i)] Suppose that $\lV$ satisfies the {\rm (V)}-boundary conditions.  
  Then there exists a closed subspace $\lW_2 \subseteq \lW^-$ of $\lW$ such that
  \[
     \lW = \lV \dot{+} \lW_2.
  \]
  Additionally, if $\lW_2$ is such a subspace and we set $\lW_1 := \lV \cap \lW_0^\perp$, so that
  $\lW = \lW_0 \dot{+} \lW_1 \dot{+} \lW_2$, with projectors $q_0,q_1,q_2$ associated to this decomposition, then the operator $M \in \mathcal{L}(\lW;\lW')$ defined with
  \[
     M = D\,(q_1 - q_2)
  \]
    satisfies the {\rm (M)}-boundary conditions and $\lV = \ker(D-M)$.  
  \item[(ii)] Conversely, let 
  $M \in \mathcal{L}(\lW;\lW')$ satisfy the {\rm (M)}-boundary conditions and set
  $\lV := \ker(D-M)$.  
  Then
  \[
     \lW_2 := \ker(D+M) \cap \lW_0^\perp
  \]
  is a closed subspace of $\lW$ contained in $\lW^-$, and we have
  \[
     \lW = \lV \dot{+} \lW_2.
  \]
\end{itemize}
\end{theorem}

\begin{remark}\label{rem:cons-M}
    Note that the construction above of $M$ from part (i) can be simplified: if $p_1,p_2$ are corresponding direct projectors with respect to the decomposition $\lW = \lV\;\dot{+}\;\lW_2$, i.e.~$p_1=q_0+q_1$ and $p_2=q_2$, then the operator $M$ can be represented as $M= D(p_1-p_2)=D(1-2p_2)$, due to $Dq_0=0$.
\end{remark}
\begin{remark}
    The existence of a closed subspace $\lW_2\subseteq \lW^-$ of $\lW$ such that $\lV\dot{+}\lW_2=\lW$ was proved in \cite[Theorem 9]{ABcpde} via rather non-trivial arguments of Kre\u\i n space theory. However,  from Remark \ref{rem:Vcond} we have $\lV\dot{+}\ker T_1=\lW$, clearly $\ker T_1$ is closed in $\lW$ and by Theorem \ref{thm:abstractFO-prop} part (ix) it is also non-positive in $\lW$. Thus, we can choose $\lW_2$ to be $\ker T_1$ (for every $\lV$ satisfying {\rm{(V)}}- boundary conditions). We state and prove this result explicitly in the following theorem and we also discuss the multiplicity of the subspaces $\lW_1$ and $\lW_2$ later.
\end{remark}

\begin{theorem}\label{tm:V-implies-M}  Let $\lV$ be a subspace satisfying {\rm{(V)}}-boundary conditions. If $p_1,p_2$ are projectors corresponding to the decomposition $\lW = \lV\;\dot{+}\;\ker T_1$, then the operator $M:= D(1-2p_2)$ satisfies {\rm{(M)}}-boundary conditions and $\lV=\ker(D-M)$.

\end{theorem}

\begin{proof}
         Let us first prove that $M$ satisfies (M1)-condition: if $u\in \lW$, then
        \begin{align*}
             \dup{\lW'}{Mu}{u}\lW&=\iscp{(1-2p_2)u}{u}
            =\iscp{(p_1-p_2)u}{(p_1+p_2)u}\;,
        \end{align*}
        where we used $p_1+p_2=\mathbbm{1}$.
        Since, $\iscp{p_1u}{p_2u}-\iscp{p_2u}{p_1u}$ has no real part (by symmetry property of indefinite inner product; see Theorem \ref{thm:abstractFO-prop}(vi)) and $\lV\subseteq \lW^+, \ker T_1\subseteq \lW^-$, we conclude that $\Re\dup{\lW'}{Mu}{u}\lW\geq 0$.
        Hence, $M$ satisfies (M1)-condition.
        The statement $\lV=\ker(D-M)$ easily follows from
        \begin{align*}
         \ker(D-M) = \ker D p_2 = \ker p_2 \cup (\ran p_2 \cap \ker D) = \lV \cup (\ker T_1\cap \lW_0) = \lV\;, 
        \end{align*}
        where we have used the definition of $M$ for the first equality, the second one is trivial, the third one follows from the definition of $p_2$ and $\ker D = \lW_0$ (Theorem \ref{thm:abstractFO-prop}.(vi)), while the last one follows from Theorem \ref{thm:abstractFO-prop}(ix).
        

    To obtain  (M2)-condition, i.e.~$\lW=\ker (D-M)+\ker(D+M)$, it is sufficient (due to the decomposition given in Theorem \ref{thm:abstractFO-prop}(xiii)) to prove that $\ker T_1\subseteq \ker (D+M)$. Since for $u\in \ker T_1$ we have $p_1u=0$, it easily follows,
    \begin{align*}
        (D+M)u=2Dp_1u=0\;,
    \end{align*}
    which completes the proof.
\end{proof}

\begin{remark}\label{rem:M-prop}
Let $\lV$ be a subspace satisfying {\rm{(V)}}-boundary conditions.
\begin{itemize} 
    \item[(i)] In the proof of the previous theorem we showed that $\ker (D-M)=\lV$ and $\ker (D+M)\supseteq \ker T_1$.
    The latter can be improved to $\ker (D+M)=\lW_0+\ker T_1$.
    In fact, the analogous identity holds for the more general construction given in Theorem \ref{thm:ABcpde-thm08}(i),
    namely
    $$
    \ker (D+M)=\lW_0+\lW_2 \,,
    $$
    with $\lW_2$ denoting the space defined in Theorem \ref{thm:ABcpde-thm08}(i).
    
    Indeed, since $\ker (D+M)$ is a subspace, from 
    $$
    \lW_0\subseteq\ker (D+M) \quad 
        \hbox{and} \quad  \lW_2\subseteq\ker (D+M) \,,
    $$
    we obtain $\lW_0+\lW_2\subseteq\ker (D+M)$.
    The first inclusion follows from 
    $\ker D=\ker M=\lW_0$, while the latter is obvious from the identity 
    $D+M = 2Dp_1$. 
    Conversely, using again $D+M = 2Dp_1$, for any $u\in \ker (D+M)$ we have $p_1u\in\lW_0$.
    Hence, writing $u=p_1u+p_2u$, we get that $u\in \lW_0+\lW_2$.

    \item[(ii)] Theorem \ref{tm:V-implies-M} ensures that for any subspace $\mathcal{V}$ satisfying the (V)-boundary conditions, there exists a corresponding operator $M$ satisfying the (M)-boundary conditions such that $\mathcal{V}=\ker(D-M)$, obtained via the construction of Theorem \ref{thm:ABcpde-thm08}(i) with $\mathcal{W}_2=\ker T_1$.
    Moreover, this operator $M$ is unique. Consequently, the assignment 
    $$
    \lV\mapsto M \,,
    $$
    is a well-defined and injective mapping.

    It is, however, not surjective: there exist operators $M$ satisfying the (M)-boundary conditions with $\mathcal{V}=\ker(D-M)$ that do not arise from this construction (see the next point and the examples following this remark).

    \item[(iii)] Let us now reflect on the construction from Theorem \ref{thm:ABcpde-thm08} in more detail: for a fixed $\lV$ satisfying (V)-boundary conditions, let us define
    \begin{align*}
        \mathbf W_\lV &:=\{\lW_2\subseteq \lW: \lW_2\subseteq \lW^- \hbox{ is a closed subspace and }\lV\dot{+}\lW_2=\lW\},\\
        \mathbf M_\lV &:= \{M\in\lL(\lW,\lW') : M \hbox{ satisfies (M)-boundary cond. and } \lV=\ker(D-M)\},
    \end{align*}
    and a mapping $\mathbf W_\lV \rightarrow \mathbf M_\lV$ defined as $\lW_2\mapsto M$ by the construction from part (i) of Theorem \ref{thm:ABcpde-thm08}. Surjectivity of such mapping follows from part (ii) of the theorem, thus all suitable $M$'s can be obtained by this construction. Regarding injectivity, we can only conclude that this mapping is injective modulo $\lW_0$, i.e.~if $M=M'\in \mathbf M_\lV$, then from the part (i) of this remark it follows that $\lW_2 + \lW_0 = \ker(D+M) = \ker(D+M') = \lW_2' + \lW_0$. 
    
    Indeed, this is the strongest conclusion we can make, since $\lW_2 + \lW_0 = \lW_2' + \lW_0$ implies $\ker(D+M) = \ker(D+M')$ (again using part (i)), which means that $M=-D=M'$ on this subspace. Since $M$ and $M'$ coincide on $\lV = \ker(D-M) = \ker(D-M')$ by the similar argument, it follows that they coincide on $\lW$ by (M2)-condition. Thus $\lW_2 + \lW_0 = \lW_2' + \lW_0$ is equivalent to $M=M'$
    
\end{itemize}
\end{remark}


We now illustrate the results and observations of this section with two examples. The first one extends Example \ref{ex:ex-semigroup}, while the second examines a partial differential operator corresponding to the stationary diffusion equation.

\begin{example}
We continue with Example \ref{ex:ex-semigroup}, 
where the spaces $\lH$, $\lW$, $\lW_0$, and the 
operators $T_0$, $T_1$, $T^\alpha$, $D$, are fixed. 

It is easy to get that the kernel of $T_1$ is given by
\begin{align*}
    \ker T_1 = \operatorname{span}\{e^{-x}\} \,,
\end{align*}
cf.~\cite[Example 3.6]{ES25}.

Let us take an arbitrary $\alpha\in\bigl(\R\cup\{\infty\}\bigr)\setminus (-1,1)$. Then $T^\alpha$ is a bijective realisation with 
signed boundary map, i.e.~$\dom T^\alpha$ satisfies 
(V)-boundary conditions. 
Thus, by Theorem \ref{lem:m-implies-v} there exists
an operator $M^\alpha\in\lL(\lW;\lW')$ satisfying 
(M)-boundary conditions. 

We shall first use the construction given in 
Theorem \ref{tm:V-implies-M} to obtain one such $M^\alpha$. 
Let $p_1^\alpha, p_2^\alpha$ be direct projectors corresponding
to the decomposition $\lW=\dom T^\alpha\;\dot{+}\;\ker T_1$.
Then, by Theorem \ref{tm:V-implies-M},
\begin{equation*}
M^\alpha= D(1-2p_2^\alpha) 
\end{equation*}
satisfies (M)-boundary conditions and corresponds to 
$\dom T^\alpha$ in the sense that 
$\ker (D-M^\alpha)=\dom T^\alpha$.
Moreover, by obtaining an explicit formula for $p_2^{\alpha}$
(see \cite[Subsection 5.3]{ES22} for similar computations),
we get:
\begin{equation*}
    \dup{\lW'}{M^\alpha u}{v}{\lW}\,=\, (u(1)-2K_u^\alpha e^{-1})v(1)
    - (u(0)-2K_u^\alpha)v(0)\,,\quad u,v\in\lW\,,
\end{equation*}
where 
$$
K_u^\alpha = \frac{u(1)-\alpha u(0)}{e^{-1}-\alpha }
    \,,
$$
and $K_u^\infty=u(0)$. 

Now, let us consider other operators that satisfy (M)-boundary conditions and correspond to the same space $\dom T^\alpha$. We will use the more general construction provided by Theorem \ref{thm:ABcpde-thm08}. 
Our aim is to find all closed $\lW_2$'s satisfying $\dom T^\alpha\dotplus \lW_2 =\lW$ and $\lW_2\subseteq \lW^-$, which by Theorem \ref{thm:ABcpde-thm08}(i) ensures the classification of all (M)-boundary conditions corresponding to a fixed $\alpha$. Since $\dim \ker T_1=1$, by Theorem \ref{thm:abstractFO-prop}(xiii) it holds $\operatorname{codim}\dom T^\alpha=1$,  implying $\dim \lW_2=1$. By Theorem \ref{thm:abstractFO-prop}(ix), we can write $\lW_2=\lW_2^{a,b}:=\operatorname{span}\{u_0 + ae^{-x}+be^x\}$ for some $u_0\in \lW_0, a,b\in \R$. By Remark \ref{rem:M-prop}(iii), it is sufficient to take $u_0=0$. Since $\lW_2^{a,b}\subseteq\lW^-$ we have that $a\neq 0$ and hence we can parameterise all such $\lW_2$'s by $r:= \frac{b}{a}$.

Let us take $r\in\R$ such that 
$$
|r|\leq  \sqrt{-\frac{\iscp{e^{-x}}{e^{-x}}}{\iscp{e^{x}}{e^{x}}}} = \frac{1}{e}
$$
and define 
$$
\lW_2^r = \operatorname{span}\left\{ e^{-x}+ re^x\right\} \,.
$$
The assumption on the parameter $r$ is precisely to ensure that
$\lW_2^r\subseteq\lW^-$. Note also that $\lW_2^0=\ker T_1$.

Clearly, the decomposition
\begin{align*}
    \lW\,=\,\dom T^\alpha\dotplus \lW_2^r
\end{align*}
holds and it is direct barring the following two exceptional cases: if $\alpha =1$, then $\lW^{e^{-1}}_2\subseteq \dom T^\alpha$ and, if $\alpha =-1$, then $\lW^{-e^{-1}}_2\subseteq \dom T^\alpha$. If the direct projections 
$p_{1}^{\alpha,r}$ and $p_2^{\alpha,r}$ are onto $\dom T^\alpha$ and $\lW_2^r$, respectively, then 
the operator $M^{\alpha,r}=D(1-2p_2^{\alpha,r})$ satisfies 
(M)-boundary conditions.
Analogously as to the case $r=0$ above, we have:
for any $u,v \in \lW$,
\begin{equation}\label{eq:M-1d-all}
    \dup{\lW'}{M^{\alpha,r}u}{v}{\lW}\,=\, 
    (u(1)-2K_{u}^{\alpha,r}(e^{-1}+re)) v(1)- (u(0)-2K_{u}^{\alpha,r}(1+r))v(0)\,,
\end{equation}
where 
$$
K_{u}^{\alpha,r} = \frac{u(1)-\alpha u(0)}{e^{-1}+re-\alpha (1+r)} \,,
$$
and $K_u^{\infty,r}=\frac{1}{1+r}u(0)$.

Note that $M^\alpha$ and $M^{\alpha,0}$ 
coincide, as expected (since $\lW_2^0=\ker T_1$).
Furthermore, by Remark \ref{rem:M-prop}(iii), we see that
\emph{all} operators satisfying (M)-boundary conditions and 
corresponding to $\dom T^\alpha$ are given by 
\eqref{eq:M-1d-all}. Indeed, one simply needs to observe that 
for any other $\lW_2\subseteq\lW^-$ such that 
$\dom T^\alpha\dot{+}\lW_2=\lW$ we have $\lW_2+\lW_0=\lW_2^r+\lW_0$
for some $r$, $r\leq  e^{-1}$. 
\end{example}

\begin{example}
Let $\Omega\subset \mathbb{R}^d$ be an open and bounded set with Lipschitz boundary $\Gamma$. Consider the following second-order partial differential equation
\begin{equation*}
    - \Delta u+u=f\;,
\end{equation*}
which can be written as a system of first-order partial
differential equations as follows:
\begin{equation}\label{eq:SDE-system}
    \left\{
\begin{array}{ll}
      \mp=-\nabla u \\
      \operatorname{div}\mp+u = f \;. \\
\end{array} 
\right.
\end{equation}
Let us define the coefficient matrices
$\mA_k=\ve_k\otimes \ve_{d+1}+\ve_{d+1}\otimes \ve_k \in \mathrm{M}_{d+1}(\R)$ for $k=1,2,...,d$, where $(\ve_1,\ve_2,...,\ve_{d+1})$ is the standard basis for $\R^{d+1}$, and let $\mB$ be the identity matrix of order $d+1$. Then the system \eqref{eq:SDE-system} can be rewritten as $T_0\vu = \vf$, where 
\begin{align*}
    T_0 \vu \;:=\; \sum_{k=1}^d \partial_k(\mA_k \vu) + \mB \vu\,,
\end{align*}
with $\vu = \begin{bmatrix}
    \mp\\ u
\end{bmatrix}$ and $\vf =\begin{bmatrix}
    0\\ f
\end{bmatrix}$. The operator $T_0$ is an abstract Friedrichs operator (see \cite{ABjde, EGC, Soni24}). Here we use our standard notation in which $T_0$ stands for minimal, and $T_1$ for the corresponding maximal operator. The graph space and the minimal space are identified as
\begin{align*}
    \lW & := L^2_{\operatorname{div}}(\Omega;\C^d)\times H^1(\Omega;\C) \\
   \mathrm{and}\quad  \lW_0& := L^2_{\operatorname{div},0}(\Omega;\C^d)\times H^1_0(\Omega;\C)\,,
\end{align*}
respectively. Here the space 
$L^2_{\operatorname{div}}(\Omega;\C^d)$ is defined as the space
of all vector-valued $L^2$ functions whose divergence is 
an (scalar) $L^2$ function as well. Then $L^2_{\operatorname{div},0}(\Omega;\C^d)$ is just the closure of
the space $C^\infty_c(\Omega;\Cr)$ in the norm of 
$L^2_{\operatorname{div}}(\Omega;\C^d)$.

Let  $\mT_0$ and $\mT_{\mnu}$ be the traces of  $H^1(\Omega;\C)$ in $H^\frac{1}{2}(\Gamma;\C)$ and $L^2_{\operatorname{div}}(\Omega;\C^d)$ in $H^{-\frac{1}{2}}(\Gamma;\C)$, respectively. Moreover, the boundary map can be characterised as,
\begin{equation}\label{eq:SDE-boundary}
   \biggl(\forall \,\,\vu := \begin{bmatrix}
        \mp\\u
   \end{bmatrix}, \vv := \begin{bmatrix}
       \mq\\v
   \end{bmatrix} \in \lW \biggr) \qquad \dup{\lW'}{D\vu}{\vv}\lW = \dup{-\frac{1}{2}}{\mT_{\mnu} \mp}{\mT_0 v}{\frac{1}{2}}+\dup{-\frac{1}{2}}{\mT_{\mnu} \mq}{\mT_0 u}{\frac{1}{2}}\,,
\end{equation}
where $\dup{-\frac{1}{2}}{\cdot}{\cdot}{\frac{1}{2}}$ denotes the duality pairing between the space $H^\frac{1}{2}(\Gamma;\C)$ and its dual $H^{-\frac{1}{2}}(\Gamma;\C)$. 
The kernel can be described as
\begin{align*}
    \ker T_1  = \{(\mp, u)^\top \in \lW: \mp = -\nabla u \ \mathrm{and} \ u = -\operatorname{div} \mp\}\,.
\end{align*}
Notice that $(\mp, u)^\top\in\ker T_1$ if and only if 
$-\Delta u + u=0$ and $\mp=-\nabla u$. Hence, $\ker T_1$ can be
parametrised by the value of $u$ on the boundary, 
i.e.~$\mT_0 u$, since the problem $-\Delta u+u=0$ together
with $\mT_0u=g$, for a fixed $g\in H^{\frac{1}{2}}(\Gamma;\C)$, has a unique solution. 

By choosing $\lV= L^2_{\operatorname{div}}(\Omega;\C^d)\times H^1_0(\Omega;\C) = \{(\mp,u)\in \lW: \mT_0 u=0\}$, the homogeneous Dirichlet boundary condition is imposed (we refer to the aforementioned references for these results).

Let us follow the construction from Theorem \ref{tm:V-implies-M}
in order to obtain an operator $M\in\lL(\lW;\lW')$ that
satisfies (M)-boundary conditions and it is associated to 
the chosen $\lV$.
For an arbitrary $u\in H^1(\Omega;\C)$, let $w_u \in H^1(\Omega; \C)$ be a unique solution of 
\begin{equation}\label{eq:hom}
    \left\{
\begin{array}{ll}
      -\Delta w_u +  w_u  =0\quad {\rm{in}}\; \Omega \\
      \mT_0 w_u = \mT_0 u \;. \\
\end{array} 
\right.
\end{equation} 
Then we know that $(-\nabla w_u, w_u)^\top \in \ker T_1$. 
For an arbitrary $\mp\in L^2_{\operatorname{div}}(\Omega;\C^d)$
we have
$$
\begin{bmatrix}\mp\\ u\end{bmatrix} =
    \begin{bmatrix}\mp +\nabla w_u\\ u-w_u\end{bmatrix}
    + \begin{bmatrix}-\nabla w_u\\ w_u\end{bmatrix} \,.
$$
Since the vector-valued functions on the right-hand side
are in $\lV$ and $\ker T_1$, respectively, the decomposition
above defines the direct projectors $p_1$ and $p_2$ from 
Theorem \ref{tm:V-implies-M}. More precisely, 
$p_2(\mp,u)^\top = (-\nabla w_u, w_u)^\top$.
Now, by Theorem \ref{tm:V-implies-M}, 
$M=D(1-2p_2)$ is a desired operator.
Using \eqref{eq:SDE-boundary}, for any $\vu:=(\mp,u)^\top\in\lW$ and $\vv:=(\mq,v)^\top\in\lW$ we get 
\begin{equation}\label{eq:SDE-M}
\begin{aligned}
\dup{\lW'}{M\vu}{\vv}\lW &= \dup{\lW'}{D(\vu-2p_2\vu)}{\vv}\lW \\
&= \dup{-\frac{1}{2}}{\mT_{\mnu} (\mp+2\nabla w_u)}{\mT_0 v}{\frac{1}{2}}+\dup{-\frac{1}{2}}{\mT_{\mnu} \mq}{\mT_0 (u-2w_u)}{\frac{1}{2}} \\
&= \dup{-\frac{1}{2}}{\mT_{\mnu} (\mp+2\nabla w_u)}{\mT_0 v}{\frac{1}{2}}-\dup{-\frac{1}{2}}{\mT_{\mnu} \mq}{\mT_0 u}{\frac{1}{2}} \\
&= \dup{-\frac{1}{2}}{\mT_{\mnu} \mp}{\mT_0 v}{\frac{1}{2}}-\dup{-\frac{1}{2}}{\mT_{\mnu} \mq}{\mT_0 u}{\frac{1}{2}}
+ 2\dup{-\frac{1}{2}}{\mT_{\mnu} \nabla w_u}{\mT_0 v}{\frac{1}{2}}\,,
\end{aligned}
\end{equation}
where in the third equality we have used that $\mT_0 w_u=\mT_0 u$.
The linear map $u\mapsto \mT_\mnu \nabla w_u$ appearing in the formula above can be recognised as the Dirichlet-to-Neumann map, corresponding to the elliptic operator $-\Delta + \mathbbm{1}$. See, for example, \cite{BE15} for recent studies of this object in a broader context.


In \cite[Section 6]{ABjde} other operators $M$ were 
provided using a rather different approach. 
More precisely, for the homogeneous Dirichlet boundary conditions, which 
are considered here, it was proved in \cite{ABjde} that, for any $\alpha\geq 0$, operators from the following family
($(\mp,u)^\top, (\mq,v)^\top\in\lW$)
\begin{equation}\label{M_alpha}
    \dup{\lW'}{M_\alpha(\mp, u)^\top}{(\mq, v)^\top}{\lW} 
    \;=\; \dup{-\frac{1}{2}}{\mT_{\mnu} \mp}{\mT_0 v}{\frac{1}{2}}-\dup{-\frac{1}{2}}{\mT_{\mnu} \mq}{\mT_0 u}{\frac{1}{2}}
    +2\alpha\int_\Gamma \mT_0u\overline{\mT_0v}\; dS \,,
\end{equation}
satisfy (M)-boundary conditions and correspond to $\lV$.


For any $\alpha\geq 0$, one can easily find a suitable subspace $\lW_2^\alpha$
which, following the construction of Theorem \ref{thm:ABcpde-thm08}, leads to the above operator $M_\alpha$: for an arbitrary $u\in H^1(\Omega;\C)$, let $v_u \in H^1(\Omega; \C)$ be a unique weak solution of the Neumann problem
\begin{equation}\label{eq:neumann}
    \left\{
\begin{array}{ll}
      -\Delta v_u +  v_u  =0\quad {\rm{in}}\; \Omega \\
      \mnu \cdot \nabla v_u = \mT_0 u \;. \\
\end{array} 
\right.
\end{equation}
Then we also have that $\Delta v_u \in L^2(\Omega;\C)$, the equation in (\ref{eq:neumann}) holds in $L^2(\Omega;\C)$, while boundary condition can be interpreted as $\mT_\mnu(\nabla v_u) = \mT_0 u$ in $H^{-1/2}(\Gamma;\C)$ \cite[pp. 380-383]{DL}.
Then a possible choice could be
$$
\lW_2^\alpha = \bigl\{(-\alpha \nabla v_u, w_u)^\top : u\in H^1(\Omega;\C) \bigr\} \,.
$$
Indeed, the decomposition
$$
\begin{bmatrix}\mp\\ u\end{bmatrix} =
    \begin{bmatrix}\mp +\alpha\nabla v_u\\ u-w_u\end{bmatrix}
    + \begin{bmatrix}-\alpha\nabla v_u\\ w_u\end{bmatrix} \,.
$$
can be used to verify that the identities $\lW=\lV+\lW_2^\alpha$ and $\ker(D+M_\alpha)=\lW_0+\lW_2^\alpha$ hold. The second equality shows that this $\lW_2^\alpha$ generates $M_\alpha$, given in (\ref{M_alpha}), by the construction from Theorem \ref{thm:ABcpde-thm08} (see also Remark \ref{rem:M-prop}). Since $\alpha\ge 0$, we easily get $\lW_2^\alpha\subseteq \lW^-$, while $\lV\cap\lW_2^\alpha=\{0\}$ follows from the fact that the zero function is the only solution of (\ref{eq:neumann}), as well as (\ref{eq:hom}), with zero boundary condition. 

It only remains to verify that $\lW_2^\alpha$ is closed in $\lW$: if $(\mp_n,u_n)$ is a sequence in $\lW_2^\alpha$ that converges to some $(\mp,u)$ in $\lW$, then there exists $\tilde{u}_n \in H^1(\Omega; \C)$ such that
\begin{equation}\label{eq:konv}
\begin{array}{ll}
      \mp_n = -\alpha \nabla v_{\tilde{u}_n} \longrightarrow \mp \quad{\rm{ in }}\; L^2_{\operatorname{div}}(\Omega;\C^d) \,,\\
      u_n= w_{\tilde{u}_n} \longrightarrow u \quad{\rm{ in }}\; H^1(\Omega;\C) \,.\\
\end{array} 
\end{equation}
From the second convergence above, the continuity of the trace operator $\mT_0$, and boundary conditions in (\ref{eq:neumann}) and (\ref{eq:hom}) we obtain that $\mT_\mnu(\nabla v_{\tilde{u}_n}) = \mT_0 ({\tilde{u}_n}) = \mT_0(w_{\tilde{u}_n})$ converges to $\mT_0(u)$ in $H^{\frac{1}{2}}(\Gamma;\C)$ and thus also in $H^{-\frac{1}{2}}(\Gamma;\C)$. An a priori estimate for (\ref{eq:neumann}) gives that $v_{\tilde{u}_n}$ converges to $v_u$ in $H^1(\Omega;\C)$, which together with the first convergence in (\ref{eq:konv}) implies that $\mp=-\alpha\nabla v_u$, by uniqueness of the limit. Similarly, from an a priori estimate for (\ref{eq:hom}) we conclude that $w_{\tilde{u}_n}$ converges to $w_u$ in $H^1(\Omega;\C)$, and thus $w_u = u$ by uniqueness of the limit. Therefore, $(\mp,u) =(-\alpha\nabla v_u, w_u)$, and thus $\lW_2^\alpha$ is closed in $\lW$.

With this, we conclude our discussion on the multiplicity of $M$-operators associated with the chosen (homogeneous Dirichlet) boundary condition, along with the corresponding subspaces $\lW_2$ used in their construction. Although the examples presented do not cover all possible $M$-operators, they serve to demonstrate that the abstract construction provided in Theorem \ref{thm:ABcpde-thm08}(i) is applicable (see also Remark \ref{rem:cons-M}). In particular, the choice $\lW_2 = \ker T_1$ (see Theorem \ref{tm:V-implies-M}) is highlighted. 
Moreover, the approach aligns with previously established $M$-operators developed through a rather different, classical method.

\end{example}

\section*{Acknowledgements}
We warmly thank M.~Waurick for insightful and stimulating discussions on this topic. 
We are also grateful to the anonymous referees for their constructive comments and suggestions that helped improve this manuscript.

This work was supported by the Croatian Science Foundation under
the projects numbers HRZZ-IP-2022-10-5181 (HOMeoS) and HRZZ-IP-2022-10-7261 (ADESO), and by the Austrian Science Fund (FWF) through the ESPRIT Programme [Grant DOI: 10.55776/ESP1299024]. For open access purposes, the authors have applied a CC BY public copyright license to any author-accepted manuscript version arising from this submission.

\end{document}